\documentclass[12pt]{amsart}

\usepackage{amssymb}
\usepackage{latexsym} 

\newtheorem{thm}{Theorem}[section]

\newtheorem{prop}[thm]{Proposition}
\newtheorem{cor}[thm]{Corollary}
\newtheorem{conj}[thm]{Conjecture}
\theoremstyle{definition}

\theoremstyle{remark}

\numberwithin{equation}{section}

\newcommand{\C}{\mathbb{C}}
\newcommand{\N}{\mathbb{N}}

\newcommand{\g}{\mathfrak{g}}
\newcommand{\frakb}{\mathfrak{b}} 
\newcommand{\frakp}{\mathfrak{p}}
\newcommand{\h}{\mathfrak{h}}
\newcommand{\frakl}{\mathfrak{l}}
\newcommand{\fraksl}{\mathfrak{sl}}
\newcommand{\frako}{\mathfrak{o}}
\newcommand{\fraksp}{\mathfrak{sp}}

\newcommand{\calB}{\mathcal{B}}

\DeclareMathOperator{\Ind}{Ind}

\numberwithin{equation}{section}

\title[Exterior powers of the reflection representation]
{Exterior powers of the reflection representation in the cohomology of Springer fibres}

\author{Anthony Henderson}
\address{School of Mathematics and Statistics\\
University of Sydney NSW 2006\\
Australia}
\email{anthonyh@maths.usyd.edu.au}

\thanks{The author's research is supported by
ARC grant DP0985184.}

\begin{document}

\begin{abstract}
Let $H^*(\calB_e)$ be the cohomology of the Springer fibre for the 
nilpotent element $e$ in a simple Lie algebra $\g$, on which the Weyl 
group $W$ acts by the Springer representation.   
Let $\Lambda^i V$ denote the $i$th exterior power of the reflection 
representation of $W$. We determine the degrees in which $\Lambda^i V$ 
occurs in the graded representation $H^*(\calB_e)$, under the assumption that
$e$ is regular in a Levi subalgebra and satisfies a certain extra condition which holds
automatically if $\g$ is of type A, B, or C.
This partially verifies a conjecture of Lehrer--Shoji, and extends
the results of Solomon in the $e=0$ case and 
Lehrer--Shoji in the $i=1$ case. The proof proceeds by showing that
$(H^*(\calB_e) \otimes \Lambda^* V)^W$ is a free exterior algebra on its
subspace $(H^*(\calB_e)\otimes V)^W$.
\end{abstract}

\maketitle

\section{Introduction}

Let $\g$ be a simple complex Lie algebra of rank $\ell$. Let $W$ denote the 
Weyl group of $\g$, and let $V$ be the reflection representation of $W$.
It is well known that the exterior powers $\Lambda^i V$, for $i=0,1,\cdots,\ell$,
are inequivalent irreducible representations of $W$, each of which is self-dual.

Let $e$ be a nilpotent element of $\g$. The Springer fibre $\calB_e$ is the variety
of Borel subalgebras of $\g$ containing $e$. Let $H^*(\calB_e)$ denote the
graded cohomology ring of $\calB_e$ with complex coefficients;
the cohomology lives solely in even degrees, so $H^*(\calB_e)$ is commutative.
We have the Springer representation of $W$ on each $H^{2j}(\calB_e)$. See
\cite{shoji}, \cite[Chapter 9]{humphreys}.

Let $s$ (depending on $e$) denote the multiplicity of the irreducible 
representation $V$ in the total representation $H^*(\calB_e)$. Let $m_1,m_2,\cdots,m_s$ be the multiset of nonnegative integers, listed in increasing order, which are the halved degrees of the occurrences
of $V$ in the graded representation $H^*(\calB_e)$. That is, we have by definition
\begin{equation} \label{exp-defn}
\sum_j \dim(H^{2j}(\calB_e)\otimes V)^W\,q^j = q^{m_1}+q^{m_2}+\cdots+q^{m_s},
\end{equation}
an equality of polynomials in the indeterminate $q$.

In the special case $e=0$, it is well known that $H^*(\calB)$ is isomorphic to
the coinvariant algebra $C^*(W)$ of $W$, that $s=\ell$, and that 
$m_1,\cdots,m_\ell$ are the exponents of $W$. More generally, if $e$ is a regular nilpotent in a Levi subalgebra of semisimple rank $r$, then it was proved by 
Lehrer and Shoji in 
\cite[Theorem 2.4]{lehrershoji} (see also \cite{spaltenstein}) that $s=\ell-r$ and that
$m_1,\cdots,m_{s}$ are the coexponents of the corresponding parabolic hyperplane
arrangement, in the sense of Orlik and Solomon. See \cite{sommerstrapa} for some other related interpretations of these coexponents. For $\g$ of classical type and
general $e$, the numbers $m_1,m_2,\cdots,m_s$ were calculated by Spaltenstein in
\cite[Propositions 1.6--1.9]{spaltenstein}.

Lehrer and Shoji conjectured that, at least in the parabolic case which they considered, 
the occurrences of each exterior power
$\Lambda^i V$ in $H^*(\calB_e)$ were also controlled in a natural way by
$m_1,m_2,\cdots,m_s$.
\begin{conj}\label{lsconj} \cite[Conjecture 8.3]{lehrershoji}
Suppose that $e$ is a regular nilpotent in a Levi subalgebra.
Then for any $i=0,1,\cdots,\ell$,
\[
\sum_j \dim(H^{2j}(\calB_e)\otimes \Lambda^i V)^W\,q^j = 
e_i(q^{m_1},q^{m_2},\cdots,q^{m_{s}}),
\]
where the right-hand side means the $i$th elementary symmetric polynomial
in $q^{m_1},q^{m_2},\cdots,q^{m_{s}}$, which is defined to
be zero if $i>s$. 
\end{conj}
\noindent
The $e=0$ case of this conjecture had already been proved by Solomon in 
\cite{solomon}; indeed, he proved the stronger statement that the algebra 
$(C^*(W) \otimes \Lambda^*V)^W$
is a free exterior algebra on $(C^*(W) \otimes V)^W$. 

In Section 2, we prove the following generalization of Solomon's result, which
implies various cases of Conjecture \ref{lsconj}.
\begin{thm} \label{exteriorthm}
Suppose that $e$ is regular in a Levi subalgebra of $\g$, 
and define $s$ and $m_1,\cdots,m_s$
as above. Also suppose that there is a parabolic
subgroup $W_K$ of $W$ such that the following two conditions hold.
\begin{enumerate}
\item There exist invariant polynomials
$f_1,f_2,\cdots,f_{s}\in (S^* V)^W$, homogeneous
of degrees $m_1+1,m_2+1,\cdots,m_{s}+1$, whose restrictions to the reflection
representation $V_K$ of $W_K$ form
a set of fundamental invariants for $W_K$.
\item The nilpotent orbit of $e$ intersects the nilradical of the parabolic subalgebra
$\frakp_{K}$ associated to $W_K$.
\end{enumerate}
\textup{(}See Section 2 for the definitions of $V_K$ and $\frakp_{K}$.\textup{)}  
Then the algebra $(H^*(\calB_e) \otimes \Lambda^* V)^W$ is a free exterior algebra on 
$(H^*(\calB_e)\otimes V)^W$. More precisely,
the natural algebra homomorphism
\[ 
\psi:\Lambda^*((H^*(\calB_e)\otimes V)^W)\to
(H^*(\calB_e) \otimes \Lambda^* V)^W
\]
is an isomorphism.
\end{thm}
\noindent
Here the domain and codomain of $\psi$ are $(\N\times\N)$-graded algebras over $\C$, 
where the $(i,j)$-components are
\[
\Lambda^i((H^{2j}(\calB_e)\otimes V)^W)\text{ and }
(H^{2j}(\calB_e) \otimes \Lambda^i V)^W
\]
respectively, and in both cases the algebra multiplication 
is graded-commutative with respect to the $\N$-grading labelled by $i$; 
the homomorphism $\psi$ is induced by
the inclusion of the subspace $(H^*(\calB_e)\otimes V)^W$ in 
$(H^*(\calB_e) \otimes \Lambda^* V)^W$. Since the graded degrees of this subspace are
$(1,m_1),(1,m_2),\cdots,(1,m_s)$, the statement that $\psi$ is an isomorphism 
implies that
\begin{equation*}
\sum_{i,j} \dim(H^{2j}(\calB_e)\otimes \Lambda^i V)^W\,t^i q^j = 
(1+tq^{m_1})(1+tq^{m_2})\cdots(1+tq^{m_s}),
\end{equation*}
which is equivalent to Conjecture \ref{lsconj}.

To illustrate the scope of the assumptions in Theorem \ref{exteriorthm},
we prove in Section 3 the following results about types A--C.
\begin{prop} \label{typeaprop}
If $\g$ is of type A, then the assumptions of Theorem \ref{exteriorthm} hold for any $e$.
\end{prop}
\begin{prop} \label{typebprop}
If $\g$ is of type B, then the assumptions of Theorem \ref{exteriorthm} hold for any $e$
which is regular in a Levi subalgebra.
\end{prop}
\begin{prop} \label{typecprop}
If $\g$ is of type C, then the assumptions of Theorem \ref{exteriorthm} hold for any $e$
which is regular in a Levi subalgebra.
\end{prop}
\noindent
Hence Conjecture \ref{lsconj} is proved in types A--C. By contrast, suppose that 
$\g$ is of type
D$_4$ and $e$ has Jordan type $(3^21^2)$ in the natural representation on $\C^8$.
Then $e$ is regular in a Levi subalgebra of type A$_2$, 
but we have $m_2=2$ and there are no 
$W$-invariant polynomials of degree $3$, so condition (1) of Theorem \ref{exteriorthm} 
cannot be satisfied.

\medskip

\textbf{Acknowledgements and historical comments.}
The main result of this paper, Theorem \ref{exteriorthm}, dates from 1997, 
when I was a student at the University of Sydney, supervised by Gus Lehrer. 
As the reader will observe, it is indebted to Lehrer's ideas, and I thank him for
his help and encouragement.
I did not publish this result at the time, since it did not prove the
motivating Conjecture \ref{lsconj} in general. Recently Eric Sommers has completed 
the proof of Conjecture \ref{lsconj} by a different method, and also removed
the assumption that $e$ is regular in a Levi subalgebra. I thank him for his
interest in my old result, and for the suggestion that it be published.
\section{Proof of Theorem \ref{exteriorthm}}
Continue the notation of the introduction. Let $\h\subset\frakb$ be a Cartan subalgebra
and Borel subalgebra of $\g$, and let $\Pi\subset\Phi^+\subset\Phi$ 
be the corresponding set of simple roots, positive roots, and roots.
We identify $W$ with the subgroup of $GL(\h)$ generated
by the simple reflections $s_\alpha$ for $\alpha\in\Pi$; the 
reflection representation $V$ of $W$ is merely $\h$ itself.

Let $J\subseteq\Pi$ be a subset of size $r$, and set $s=\ell-r$. We have a Levi subalgebra
$\frakl_J$ and parabolic subalgebra $\frakp_J$ containing $\h$ and $\frakb$
respectively, a parabolic subsystem $\Phi_J$ of $\Phi$,
and a parabolic subgroup $W_J$ of $W$. Define
\[ V^J=\bigcap_{\alpha\in J}\ker(\alpha)
=V^{W_J}. \]
We write $V_J$ for the unique $W_{J}$-invariant complement to
$V^{J}$ in $V$, which is the reflection representation of $W_J$.
Note that $\dim V^J=s$ and $\dim V_J=r$.
Let $\mathcal{A}^J$ and $\mathcal{A}_{J}$ be the hyperplane arrangements in $V^J$
and $V_J$ respectively induced by the root
hyperplanes in $V$.

We assume for the remainder of the section that
\begin{equation} \label{ass1}
\text{$e$ is parabolic of type $J$,}
\end{equation} 
meaning that the orbit of $e$ contains the 
regular nilpotent elements of $\frakl_J$. As mentioned in the introduction, Lehrer and
Shoji proved in this case that
\begin{equation}
\sum_j \dim(H^{2j}(\calB_e)\otimes V)^W\,q^j = q^{m_1}+q^{m_2}+\cdots+q^{m_{s}},
\end{equation}
where $m_1,\cdots,m_{s}$ are the coexponents of the arrangement $\mathcal{A}^J$.
(See \cite[Theorem 2.4]{lehrershoji}; the missing case in type D
is covered by the results of Spaltenstein \cite{spaltenstein}.) 

An important special feature of the parabolic case is Lusztig's Induction Theorem
for Springer representations.
\begin{thm} \label{lusztig}
\cite{lusztig}
Under the assumption \eqref{ass1}, the representation of $W$ on
$H^*(\calB_e)$, neglecting the grading, 
is isomorphic to the induction $\Ind_{W_J}^W(\C)$ of the trivial
representation of $W_J$.
\end{thm}
\begin{cor}
Under the assumption \eqref{ass1}, we have 
\[ \dim(H^*(\calB_e) \otimes \Lambda^* V)^W = 2^{s}, \]
so the domain and codomain of $\psi$ have the same dimension.
\end{cor}
\begin{proof}
By Frobenius Reciprocity, we know that
\begin{equation}
\dim(\Ind_{W_J}^W(\C)\otimes \Lambda^* V)^W=\dim (\Lambda^* V)^{W_J}.
\end{equation}
From the fact that $(\Lambda^* V)^{s_\alpha}=\Lambda^*(\ker(\alpha))$ for
all $\alpha\in J$ one deduces $(\Lambda^* V)^{W_J}=\Lambda^*(V^J)$,
and the Corollary follows.
\end{proof}

So to prove Theorem \ref{exteriorthm},
it suffices to show that the homomorphism $\psi$ is injective.
Since the domain is an exterior algebra, this will follow if we can show that 
\begin{equation} \label{nonzero1}
\psi(\Lambda^{s}((H^*(\calB_e)\otimes V)^W))\neq 0.
\end{equation}
To be more concrete,
let $v_1,\cdots,v_\ell$ be a basis of $V$, and let $c_{i,j}\in H^{2m_i}(\calB_e)$,
for $1\leq i\leq s$, $1\leq j\leq \ell$, be such that
\[
\sum_{j=1}^n c_{1,j}\otimes v_j,\
\sum_{j=1}^n c_{2,j}\otimes v_j,\ \cdots,\
\sum_{j=1}^n c_{s,j}\otimes v_j
\]
is a basis of $(H^*(\calB_e)\otimes V)^W$. Then by definition of $\psi$,
\eqref{nonzero1} is equivalent to the statement that for some
sequence $1\leq j_1<j_2<\cdots<j_{s}\leq n$,
\begin{equation} \label{nonzero2}
\det\begin{pmatrix}
c_{1,j_1}&c_{1,j_2}&\cdots&c_{1,j_{s}}\\
c_{2,j_1}&c_{2,j_2}&\cdots&c_{2,j_{s}}\\
\vdots&\vdots&\ddots&\vdots\\
c_{s,j_1}&c_{s,j_2}&\cdots&c_{s,j_{s}}
\end{pmatrix}\neq 0.
\end{equation} 
To handle such a determinant, we need a better grasp of the cohomology ring
$H^*(\calB_e)$ in which it lives.

Using the $W$-equivariant isomorphism of $V$ with its dual, we will identify
the symmetric algebra $S^* V$ with the ring of polynomial functions on $V$. It is well
known that the invariant subring $(S^* V)^W$ is freely generated by $\ell$ homogeneous
polynomials, called fundamental invariants for $W$.
The coinvariant algebra $C^*(W)$ of $W$ is the quotient $S^* V/I$, where $I$ is the
ideal generated by these fundamental invariants.

Now there is a canonical (degree-doubling) 
$W$-equivariant algebra homomorphism $S^* V\to H^*(\calB)$ which identifies
$H^*(\calB)$ with $C^*(W)$ (see \cite[\S5]{shoji}). Composing this with the
natural homomorphism $H^*(\calB)\to H^*(\calB_e)$, we obtain a $W$-equivariant
homomorphism $\varphi:S^*V\to H^*(\calB_e)$ (see \cite[Theorem 1.1]{hottaspringer}). 
Note that the image of $\varphi$
is contained in the subspace $H^*(\calB_e)^{A(e)}$ of invariants for
the component group of the centralizer of $e$ in the adjoint group of $\g$; in
particular, $\varphi$ is not surjective in general. However, it may happen that
the induced map $(S^*V\otimes V)^W\to(H^*(\calB_e)\otimes V)^W$
is surjective even if $\varphi$ is not.
We will see that this occurs under the
assumptions of Theorem \ref{exteriorthm}, which means that a calculation with
polynomials on $V$ suffices to prove \eqref{nonzero2}. 

Henceforth we let $K\subseteq\Pi$ be a subset satisfying conditions
(1) and (2) of Theorem \ref{exteriorthm}. Note that condition (1) entails $|K|=s$.
Choose a basis $v_1,v_2,\cdots,v_\ell$ of $V$ such that $v_1,\cdots,v_{s}$ is a basis
of $V_K$.
Since the exterior derivative $S^*V\to S^*V\otimes V:f\mapsto \sum 
\frac{\partial f}{\partial v_j}\otimes v_j$ is $W$-equivariant, we have the
following $s$ elements of $(H^*(\calB_e)\otimes V)^W$:
\[
\sum_{j=1}^n \varphi\!\left(\frac{\partial f_1}{\partial v_j}\right)\otimes v_j,\
\sum_{j=1}^n \varphi\!\left(\frac{\partial f_2}{\partial v_j}\right)\otimes v_j,\ \cdots,\
\sum_{j=1}^n \varphi\!\left(\frac{\partial f_{s}}{\partial v_j}\right)\otimes v_j.
\]
We can prove simultaneously that these form a basis of $(H^*(\calB_e)\otimes V)^W$,
and that \eqref{nonzero2} holds (with $j_i=i$), by proving the single fact that 
$\varphi(\Delta)\neq 0$,
where 
\[
\Delta=\det\begin{pmatrix}
\frac{\partial f_1}{\partial v_1}&\frac{\partial f_1}{\partial v_2}&\cdots
&\frac{\partial f_1}{\partial v_{s}}\\
\frac{\partial f_2}{\partial v_1}&\frac{\partial f_2}{\partial v_2}&\cdots
&\frac{\partial f_2}{\partial v_{s}}\\
\vdots&\vdots&\ddots&\vdots\\
\frac{\partial f_{s}}{\partial v_1}&\frac{\partial f_{s}}{\partial v_2}&\cdots
&\frac{\partial f_{s}}{\partial v_{s}}\\
\end{pmatrix}.
\]
Since $v_1,\cdots,v_{s}$ span the reflection representation of $W_K$, we have
$w\Delta=\varepsilon(w)\Delta$ for all $w\in W_K$. This forces
$\Delta$ to be divisible by the polynomial $\pi_K=\prod_{\beta\in\Phi_K^+}\beta$.
Condition (1) tells us that on restricting to $V_K$,
$\Delta$ becomes the Jacobian matrix of the fundamental invariants of $W_K$,
which is well known to be a nonzero scalar multiple of the restriction to $V_K$ of
$\pi_K$ (see \cite{steinberg}). So $\Delta$ is a nonzero scalar multiple of $\pi_K$,
and it suffices to prove that $\varphi(\pi_K)\neq 0$. 

By condition (2), we may suppose that $e$ lies in the nilradical of $\frakp_K$.
Then any Borel subalgebra contained in $\frakp_K$ must contain $e$, so
we have an inclusion $\calB^K\hookrightarrow\calB_e$, where $\calB^K$ denotes the variety
of Borel subalgebras contained in $\frakp_K$, which can be identified with the
flag variety of $\frakl_K$. Hence it suffices to prove that
$\pi_K$ is not in the kernel of the composition $S^*V\to H^*(\calB)\to H^*(\calB^K)$.
But this composition is the canonical homomorphism identifying $C^*(W_K)$ with
$H^*(\calB^K)$, which maps $\pi_K$ to a generator of the top-degree cohomology of 
$\calB^K$ (see \cite[Proposition 1.4]{hottaspringer}, which uses exactly this argument
in the case when $e$ lies in the Richardson orbit of $\frakp_K$). This completes
the proof of Theorem \ref{exteriorthm}.
\section{Proofs of Propositions \ref{typeaprop}--\ref{typecprop}}

\noindent
\textit{Proof of Proposition \ref{typeaprop}.}\
In type A, we have $\g=\fraksl_{\ell+1}$ and $W=S_{\ell+1}$. 
The simple roots are given by $\alpha_i=x_i-x_{i+1}$, where
$x_1,\cdots,x_{\ell+1}$ denote coordinate functions on $V$ which are permuted by $W$,
and which satisfy $x_1+\cdots+x_{\ell+1}=0$. The orbit of the nilpotent $e$ is
determined by its Jordan type
$\lambda$, which can be an arbitrary partition of $\ell+1$: the element $e$ is then automatically parabolic of type
\[ J=\Pi\setminus\{\alpha_{\lambda_1+\lambda_2+\cdots+\lambda_i}\,|\,1\leq 
i<\ell(\lambda)\},\text{ so }r=\ell-\ell(\lambda)+1,\,s=\ell(\lambda)-1. \]
In this case, $\mathcal{A}^J$ is an arrangement of type A$_{s}$,
so $m_i=i$ for $i=1,\cdots,s$. We define $f_i$ to be
the usual $S_{\ell+1}$-invariant polynomial of degree $i+1$, namely the 
$(i+1)$th elementary symmetric polynomial in $x_1,\cdots,x_{\ell+1}$. Let
\[ K=\{\alpha_{r+1},\alpha_{r+2},\cdots,
\alpha_{\ell}\}. \]
Then $V_K$ is defined by the equations $x_i=0$ for all $i\leq r$, so
the restrictions of the $f_i$'s to $V_K$ are the 
elementary symmetric polynomials in $x_{r+1},\cdots,
x_{\ell+1}$, and condition
(1) is satisfied. With this $K$, condition (2) amounts to saying that there is 
a partial flag in the natural representation,
\[ 0=U_0\subset U_1\subset U_2\subset\cdots\subset U_{r}\subset\C^{\ell+1}, \]
such that $eU_i\subseteq U_{i-1}$ and $\dim U_i=i$ for
all $i\geq 1$, and $e(\C^{\ell+1})\subseteq U_r$. This is true: we can take
$U_r=e(\C^{\ell+1})$, and then let $U_1\subset\cdots\subset U_{r}$ be
any $e$-stable complete flag in $U_r$.
\hfill$\Box$

\medskip

\noindent
\textit{Proof of Proposition \ref{typebprop}.}\
In type B, we have $\g=\frako_{2\ell+1}$. The simple roots are given by
$\alpha_i=x_i-x_{i+1}$ for $i\leq \ell-1$ and $\alpha_\ell=x_\ell$,
where $x_1,\cdots,x_\ell$ denote coordinate functions on
$V$ which are permuted up to sign by $W$. We are assuming that $e$ is parabolic
of type $J$, for some subset $J\subseteq\Pi$.
There is some nonnegative integer $m\leq \ell$ and composition 
$\lambda$ of $\ell-m$ such that
\[ J=\Pi\setminus\{\alpha_{\lambda_1+\lambda_2+\cdots+\lambda_i}\,|\,1\leq 
i\leq\ell(\lambda)\},\text{ so }r=\ell-\ell(\lambda),\,s=\ell(\lambda). \]
The orbit containing the regular nilpotents in $\frakl_J$ consists of those 
whose Jordan type in the natural
representation $\C^{2\ell+1}$ is the partition obtained by rearranging 
$\lambda_1,\lambda_1,\lambda_2,\lambda_2,\cdots,
\lambda_{s},\lambda_{s},2m+1$. 
Grouping the Jordan blocks together as appropriate, we get
a direct sum decomposition $\C^{2\ell+1}=U\oplus U'\oplus U''$ where
$U$ and $U'$ are $(\ell-m)$-dimensional $e$-stable isotropic subspaces on each of
which the Jordan type of $e$ is the partition obtained by rearranging $\lambda$, 
$U''$ is a $(2m+1)$-dimensional $e$-stable subspace on which $e$ has a single
Jordan block, and $U''$ is
perpendicular to $U\oplus U'$ for the nondegenerate symmetric form.

The arrangement $\mathcal{A}^J$
is of type B$_{s}$, so $m_i=2i-1$ for $i=1,\cdots,s$.
We define $f_i$ to be the usual $W$-invariant polynomial of degree $2i$, namely
the $i$th elementary symmetric polynomial in $x_1^2,\cdots,x_\ell^2$. Let
\[ K=\{\alpha_{r+1},\alpha_{r+2},\cdots,
\alpha_{\ell}\}. \]
Then $V_K$ is defined by the equations $x_i=0$ for $i\leq r$, so the
restrictions of the $f_i$'s to $V_K$ are the elementary symmetric polynomials in
$x_{r+1}^2,\cdots,x_\ell^2$, and condition (1) is satisfied.
With this $K$, condition (2) amounts to saying that there is an isotropic partial
flag in the natural representation,
\[ 0=U_0\subset U_1\subset\cdots\subset U_{r}\subset\C^{2\ell+1}, \]
such that $eU_i\subseteq U_{i-1}$ and $\dim U_i=i$ 
for all $i\geq 1$, and $e(U_{r}^\perp)\subseteq
U_{r}$.  This is true: if we define
$U_{r}$ to be
$e(U)\oplus U'''$ where $U'''$ is the unique $e$-stable $m$-dimensional subspace of $U''$, then $U_r^\perp=U\oplus\ker(e|_{U'})\oplus e^{-1}(U''')$, so $e(U_r^\perp)=U_r$. 
We can then let $U_1\subset\cdots\subset U_{r}$ be any $e$-stable complete
flag in $U_{r}$. 
\hfill$\Box$

\medskip

\noindent
\textit{Proof of Proposition \ref{typecprop}.}\
In type C, we have $\g=\fraksp_{2\ell}$ with natural representation $\C^{2\ell}$; the
description of the simple roots is the same as in type B, except that 
$\alpha_\ell=2x_\ell$.
The proof is identical to that of Proposition \ref{typebprop}, but with
the dimension of $U''$
(which equals the size of a Jordan block of $e$) changed from $2m+1$ to $2m$,
so in the last step one has $U_r^\perp=U\oplus\ker(e|_{U'})\oplus U'''$, still
implying $e(U_r^\perp)\subseteq U_r$. 
\hfill$\Box$

\end{document}